\documentclass[12pt]{amsart}
\usepackage{amsmath,amssymb,amsthm,amscd,url}
\newtheorem{thm}{Theorem}
\newtheorem*{thmnonum}{Theorem}

\newtheorem{cor}[thm]{Corollary}
\newtheorem*{rem}{Remark}
\newtheorem*{ack}{Acknowledgements}
\newtheorem*{ex}{Example}

\newcommand{\SL}{{\rm SL}}

\newcommand{\C}{\mathbb{C}}
\renewcommand{\H}{\mathbb{H}}

\newcommand{\Z}{\mathbb{Z}}

\newcommand{\Aut}{{\rm Aut}}

\newcommand{\legen}[2]{\genfrac{(}{)}{}{}{#1}{#2}}
\newcommand{\Gen}{{\rm Gen}}

\usepackage{color}
\begin{document}

\title[Quadratic forms representing all integers coprime to $3$]{Quadratic forms representing all integers coprime to $3$}
\author{Justin DeBenedetto}
\address{Department of Computer Science and Engineering, University of
Notre Dame, Notre Dame, IN 46556}
\email{jdebened@nd.edu}
\author{Jeremy Rouse}
\address{Department of Mathematics and Statistics, Wake Forest University,
  Winston-Salem, NC 27109}
\email{rouseja@wfu.edu}

\begin{abstract}
Following Bhargava and Hanke's celebrated 290-theorem, we prove a universality
theorem for all positive-definite integer-valued quadratic forms that
represent all positive integers coprime to $3$. In particular, if
a positive-definite quadratic form represents all positive integers coprime
to $3$ and $\leq 290$, then it represents all positive integers coprime to $3$.
We use similar methods to those used by Rouse to prove (assuming GRH)
that a positive-definite quadratic form representing every odd integer
between $1$ and $451$ represents all positive odd integers.
\end{abstract}

\maketitle

\section{Introduction and Statement of Results}
\label{intro}

The study of which integers are represented by certain quadratic forms
dates back to the time of Diophantus in the 3\textsuperscript{rd} century. Building on this
work, Fermat classified the positive integers that can be written
as a sum of two integer squares, i.e., the numbers represented by
$x^{2} + y^{2}$. In 1770, Lagrange proved that every positive integer
can be written as a sum of four squares.

In 1916, Ramanujan \cite{Ramanujan} gave a list of $55$ quadratic forms
of the form $ax^{2} + by^{2} + cz^{2} + dw^{2}$ and claimed that these
quadratic forms are the only diagonal forms in four variables that represent
all positive integers. Dickson \cite{Dickson} proved Ramanujan's claim (modulo the error that Ramanujan had listed one form that fails to
represent $15$).

We say that a positive-definite quadratic form $Q(\vec{x})$ is an
\emph{integer-matrix} form if $Q(\vec{x}) = \vec{x}^{T} A \vec{x}$,
where $A$ is a matrix with integer entries. We say that $Q$ is
\emph{integer-valued} if
$Q(\vec{x}) = \frac{1}{2} \vec{x}^{T} A \vec{x}$ where $A$ is a matrix
with integer entries and even diagonal entries.

In her Ph.D. thesis, Willerding \cite{Willerding} classified \emph{universal}
integer-matrix quaternary forms, those that represent all positive integers.
In 1993, Conway and Schneeberger proved the following theorem giving
a nice classification of universal forms in any number of variables
(see \cite{SchneebergerThesis}).
\begin{thmnonum}[``The 15-Theorem'']
\label{CS}
A positive-definite integer-matrix quadratic form is universal if and only if
it represents the numbers
\[
  1, 2, 3, 5, 6, 7, 10, 14, \text{ and } 15.
\]
\end{thmnonum}

This theorem was elegantly reproven by Bhargava in 2000 (see
\cite{Bhar}).  Bhargava's approach is to work with integral lattices,
and to classify escalator lattices - lattices that must be inside any
lattice whose corresponding quadratic form represents all positive
integers. As a consequence, Bhargava was able to correct some errors in
Willerding's work.

Bhargava's approach is quite general. Indeed, he has proven that for
any infinite set $S$, there is a unique minimal finite subset $S_{0}$
of $S$ so that any positive-definite integral quadratic form
represents all numbers in $S$ if it represents the numbers in
$S_{0}$. Here the notion of integral quadratic form can mean either
integer-matrix or integer-valued (and the set $S_{0}$ depends on which
notion is used).

While working on the 15-Theorem, Conway and Schneeberger were led to conjecture
that every integer-valued quadratic form that represents the
positive integers between 1 and 290 must be universal. Bhargava and
Hanke's celebrated 290-Theorem proves this conjecture (see \cite{BH}).
Their result is the following.
\begin{thmnonum}[``The 290-Theorem'']
If a positive-definite integer-valued quadratic form
represents the 29 integers
\begin{align*}
 & 1, 2, 3, 5, 6, 7, 10, 13, 14, 15, 17, 19, 21, 22, 23, 26, 29,\\
 & 30, 31, 34, 35, 37, 42, 58, 93, 110, 145, 203, \text{and } 290,
\end{align*}
then it represents all positive integers.
\end{thmnonum}
They also show that every one of the twenty-nine integers above is necessary.
Indeed, for every integer $t$ on this list, there is a positive-definite
integer-valued quadratic form that represents every positive integer except $t$.
As a consequence of the 290-Theorem, they are able to prove that there are
exactly $6436$ universal integer-valued quaternary quadratic forms.

Bhargava has shown that if $Q$ is a positive-definite integer
matrix form that represents the integers from $1$ up to $33$, then $Q$
represents all positive odd integers.  In \cite{Rouse}, the second
author proved (assuming the generalized Riemann hypothesis) that an integer-valued form that represents the odd integers $\leq 451$ must represent all odd integers. Recently, Barowsky, Damron, Mejia, Saia, Schock and Thomspon gave a classification
of the possible sets $\{m,n\}$ of exceptions for an integer-matrix form with exactly two exceptions (see \cite{BDMSST}). 
The goal of the present paper is to let $S$ be the set of positive integers 
coprime to $3$, and to compute the minimal subset $S_{0}$ of integers that an integer-valued form $Q$ must represent in order to represent
everything in $S$. Our main result is the following.

\begin{thm}[The CCXC Theorem]
\label{main}
If a positive-definite integer-valued quadratic form represents
the following 31 integers
\begin{align*}
  & 1, 2, 5, 7, 10, 11, 13, 14, 17, 19, 22, 23, 26, 29, 31, 34, 35,\\
  & 37, 38, 46, 47, 55, 58, 62, 70, 94, 110, 119, 145, 203, 290,
\end{align*}
then it represents all positive integers coprime to $3$.
\end{thm}

Here are two corollaries.

\begin{cor}
\label{necessary}
For every single one of the positive integers $t$ in the above list,
there is a positive-definite integer-valued quadratic form $Q$ that represents
every positive integer coprime to $3$ except $t$.
\end{cor}

\begin{cor}
\label{corollary}
If a positive-definite integer-matrix quadratic form represents
the following integers
\[
  1, 2, 5, 7, 10, 11, 14, 19, 22, 31, 35,
\]
then it represents all positive integers coprime to $3$.
\end{cor}

To prove the CCXC Theorem, we must determine exactly which positive, squarefree
integers that are coprime to $3$ are represented by a collection of $9611$
quaternary quadratic forms. Any form that represents all positive integers
coprime to $3$ must represent either one of 11 regular ternary quadratic forms,
or one of the quaternary forms on this list.

To analyze the quaternary forms, we use a combination of four
methods. These methods are the same ones used in the proof of the 451-theorem of \cite{451paper}. The first method checks to see if a given quaternary
represents any of the $11$ regular ternaries mentioned above. If so,
it represents all positive integers coprime to $3$. This method
succeeds for $646$ of the $9611$ quaternary forms.

The second method attempts to find, given the integer lattice $L$
corresponding to $Q$, a regular ternary sublattice $K$ so that
$K \oplus K^{\perp}$ locally represents everything coprime to $3$.
We make use of the classification of regular ternary quadratic
forms due to Jagy, Kaplansky, and Schiemann \cite{JKS}. This method
is successful for $3631$ of the remaining forms.

The last two methods rely on the theory of modular forms. For a
positive-definite quaternary form $Q$, we define
$\chi(n) = \legen{\det(A)}{n}$ to be the usual Kronecker
character. Then the theta series
\[
  \theta_{Q}(z) = \sum_{n=0}^{\infty} r_{Q}(n) q^{n}, \quad q = e^{2 \pi i z}
\]
is a modular form of weight $2$, level $N$ and character $\chi$. We
can decompose $\theta_{Q}(z)$ as
\begin{align*}
  \theta_{Q}(z) &= E(z) + C(z)\\
  &= \sum_{n=0}^{\infty} a_{E}(n) q^{n} + \sum_{n=1}^{\infty} a_{C}(n) q^{n}.
\end{align*}
A lower bound on $a_{E}(n)$ is given in Theorem 5.7 of \cite{Hanke}
and shows that
\[
  a_{E}(n) \geq C_{E} n \prod_{\substack{p | n \\ \chi(p) = -1}} \frac{p-1}{p+1}
\]
for some some constant $C_{E}$, depending on $Q$, provided $n$ is squarefree
and locally represented by $Q$. We may decompose the form $C(z)$ into
a linear combination of newforms (and the images of newforms under $V(d)$).
It is known that the $n$th Fourier coefficient of a newform
of weight $2$ is bounded by $d(n) n^{1/2}$ by the Deligne bound. Thus, there is
a constant $C_{Q}$ so that
\[
  |a_{C}(n)| \leq C_{Q} d(n) n^{1/2}.
\]
If we can compute or bound the constants $C_{E}$ and $C_{Q}$, we can determine
the squarefree integers represented by $Q$ via a finite computation.

Finally, the third method explicitly computes the constant $C_{Q}$
by doing extensive exact linear algebra using Magma. This method handles
$2267$ forms.

The fourth method gives an upper bound on $C_{Q}$ without explicitly computing
it. This method may only be used when $\det(A)$ is a fundamental discriminant.
The Petersson inner product is defined for $f, g \in S_{2}(\Gamma_{0}(N), \chi)$ by
\[
  \langle f, g \rangle = \frac{3}{\pi [\SL_{2}(\Z) : \Gamma_{0}(N)]}
  \iint_{\H/\Gamma_{0}(N)} f(x+iy) \overline{g}(x+iy) \, dx \, dy.
\]
If one can compute an upper bound on $\langle C(z), C(z) \rangle$
and a lower bound on $\langle g_{i}(z), g_{i}(z) \rangle$ for each newform
$g_{i}(z)$ in $S_{2}(\Gamma_{0}(N), \chi)$, one can derive an upper bound
on $C_{Q}$. We use the same machinery from \cite{451paper} to accomplish
these things. We use this method for the remaining $3067$ forms.

One different feature of the present work is the use of the exact
formulas due to Yang \cite{Yang} for the local densities appearing in
the formula for $a_{E}(n)$. These provide for efficient computation.

An outline of the paper is as follows. In Section~\ref{back} we review
background about quadratic forms and modular forms. In Section~\ref{escalate}
we describe the theory of escalator lattices and our use of it. In Section~\ref{regtern} we describe in detail our first two methods relying on properties
of ternary quadratic forms. In Section~\ref{modform} we describe
our first modular form based method (the only method used by Bhargava
and Hanke in \cite{BH}), and in Section~\ref{petnorm} we describe our
second modular form method (pioneered in \cite{451paper}). Finally,
in Section~\ref{pfs} we prove Theorem~\ref{main}, Corollary~\ref{necessary}, and Corollary~\ref{corollary}.

\begin{ack}
The authors used the computer software package Magma \cite{Magma} version 2.21-1
extensively for the computations. Magma scripts and log files from the computations
done are available at {\normalfont \url{http://users.wfu.edu/rouseja/CCXC/}}. This work represents the master's thesis
of the first author completed at Wake Forest University in the spring of 2015.
\end{ack}

\section{Background}
\label{back}

If $Q$ is an integer-valued quadratic form in $r$ variables, then $Q(\vec{x}) = \frac{1}{2} \vec{x}^{T} A \vec{x}$
for some $r \times r$ matrix $A$. We say that the discriminant of $Q$ is $\det(A)$, and the level of $Q$ is the
smallest positive integer $N$ so that $NA^{-1}$ has integer entries and even diagonal entries.

If $N$ is a positive integer, define $\Gamma_{0}(N) = \left\{ \begin{bmatrix} a & b \\ c & d \end{bmatrix} \in \SL_{2}(\Z) :
N | c \right\}$. If $k$ is a positive even integer, let $M_{k}(\Gamma_{0}(N), \chi)$ denote the $\C$-vector space of
modular forms $f$ so that
\[
  f\left(\frac{az+b}{cz+d}\right) = \chi(d) (cz+d)^{k} f(z).
\]
Let $S_{k}(\Gamma_{0}(N), \chi)$ denote the subspace of cusp forms. The operator $V(d)$ is defined
by $\sum a(n) q^{n} \mapsto \sum a(n) q^{dn}$ and maps $S_{k}(\Gamma_{0}(N), \chi)$ to
$S_{k}(\Gamma_{0}(dN), \chi)$. The \emph{old subspace} of $S_{k}(\Gamma_{0}(N), \chi)$
is defined to be the span of the images of $V(e) : S_{k}(\Gamma_{0}(N/d), \chi) \to S_{k}(\Gamma_{0}(N), \chi)$
where $e$ runs over divisors of $d$, and $d$ runs over divisors of $N$. The \emph{new subspace}
of $S_{k}(\Gamma_{0}(N), \chi)$ is defined to be the orthogonal complement of the old subspace under
the Petersson inner product. This new subspace is spanned by \emph{newforms} - Hecke eigenforms lying
in the new subspace that are normalized so the Fourier coefficient of $q = e^{2 \pi i z}$ is $1$.

If $Q$ is a positive-definite integer-valued quadratic form in $r$ variables, let
$r_{Q}(n) = \{ \vec{x} \in \Z^{r} : Q(\vec{x}) = n \}$. The \emph{theta series} of $Q$ is
\[
  \theta_{Q}(z) = \sum_{n=0}^{\infty} r_{Q}(n) q^{n}.
\]
If $N$ is the level of $Q$ and $r$ is even, then $\theta_{Q} \in M_{r/2}(\Gamma_{0}(N), \chi_{D})$ (see
Theorem 10.8 of \cite{Iwa}). Here $\chi_{D}(\cdot) = \legen{(-1)^{r/2} \det A}{\cdot}$ is the usual Kronecker symbol.
As noted in Section~\ref{intro}, $\theta_{Q}(z)$ has a decomposition
$E(z) + C(z) = \sum_{n=0}^{\infty} a_{E}(n) q^{n} + \sum_{n=1}^{\infty} a_{C}(n) q^{n}$, into an Eisenstein series and a cusp form.

We can associate a lattice $L$ to a positive-definite integer-valued quadratic form $Q$ by letting $L = \Z^{r}$
and defining an inner product on $L$ by setting
\[
  \langle \vec{x}, \vec{y} \rangle = \frac{1}{2} \left(Q(\vec{x} + \vec{y}) - Q(\vec{x}) - Q(\vec{y})\right).
\]
We have that $\langle \vec{x}, \vec{x} \rangle = Q(\vec{x})$ is integral, but arbitrary inner products $\langle \vec{x}, \vec{y} \rangle$
need not be integral. If $Q = \frac{1}{2} \vec{x}^{T} A \vec{x}$,
we say that $A$ is the \emph{Gram matrix} of $L$. We will move freely
between a quadratic form and its corresponding lattice, and use adjectives
that apply to quadratic forms to refer to lattices and vice versa.

For a prime $p$, let $\Z_{p}$ denote the ring of $p$-adic integers. We say that a positive-definite form
$Q$ \emph{locally represents} an integer $m$ if $m > 0$ and for all primes $p$, there is some $\vec{x} \in \Z_{p}^{r}$
so that $Q(\vec{x}) = m$. If $Q$ is fixed, we let $\Gen(Q)$ denote the finite collection of positive-definite integral
forms $R$ so that $R$ is equivalent to $Q$ over $\Z_{p}$ for all primes $p$. By work of Siegel \cite{Siegel}, we have that
\[
  \sum_{n=0}^{\infty} a_{E}(n) q^{n}
    = \frac{\sum_{R \in \Gen(Q)} \theta_{R}(z)/\# \Aut(R)}{\sum_{R \in \Gen(Q)} 1/\# \Aut(R)}
    = \sum_{n=0}^{\infty} \left(\prod_{p \leq \infty} \beta_{p}(n)\right) q^{n},
\]
where $\beta_{p}(n)$ is the \emph{local density} associated to $p$, $Q$ and $n$. It follows from this formula
that if $|\Gen(Q)| = 1$, then $Q$ represents every integer $n$ that is locally represented by $Q$. If $Q$ is a quadratic form
that represents all positive integers $n$ that are locally represented, we say that $Q$ is \emph{regular}.

\section{Escalators}
\label{escalate}

Fix a set $S$ of positive integers. Given a quadratic form $Q$ with corresponding lattice $L$,
we say that an element $t \in S$ is an \emph{exception} for $Q$ if $Q$ does not represent $t$. We call
the smallest exception the \emph{truant} of $Q$. If $Q$ is a quadratic form with truant $t$ with corresponding lattice $L$,
an \emph{escalation} of $L$ is a lattice $L'$ generated by $L$ and a vector of norm $t$. We say that $L$ (or $Q$) is
\emph{relatively universal} if it represents everything in $S$ (or equivalently, if it has no truant). An \emph{escalator lattice}
is a lattice obtained by repeated escalation of the unique zero-dimensional lattice.

Write $S = \{ t_{1}, t_{2}, t_{3}, \ldots \}$ with $t_{i} < t_{j}$ if $i < j$. If $L$ is a relatively universal lattice, then it contains a vector of norm $t_{1}$,
and hence an escalator lattice $L_{1}$ generated by $t_{1}$. If $L_{1}$ is not relatively universal, then there is a vector in $L$ with norm equal to
the truant of $L_{1}$. Then $L$ must contain some escalation $L_{2}$ of $L_{1}$. Continuing in this way, we get a sequence
of escalator lattices $L_{1} \subseteq L_{2} \subseteq L_{3} \subseteq \cdots \subseteq L$ with the property that $L_{i}$
represents (at least) the first $i$ elements of $S$. Since $L \cong \Z^{r}$ is a noetherian $\mathbb{Z}$-module,
this ascending chain of lattices stabilizes. Thus, there is some $L_{n} \subseteq L$ that is a relatively universal escalator lattice.

We are concerned with the case that $S = \{ n \geq 1 : \gcd(n,3) = 1 \}$.
To prove Theorem~\ref{main}, we begin by escalating the zero-dimensional lattice. We obtain the $1$-dimensional lattice with
Gram matrix $A = \begin{bmatrix} 2 \end{bmatrix}$ corresponding to the quadratic form $Q(\vec{x}) = \frac{1}{2} \vec{x}^{T} A \vec{x} = x^{2}$.
This has truant $2$ and (by the Cauchy-Schwarz inequality), its escalations are those lattices with Gram matrices
\[
  \begin{bmatrix}
  2 & 0\\
  0 & 4 \end{bmatrix}, \quad \begin{bmatrix}
  2 & 1\\
  1 & 4 \end{bmatrix}, \quad \begin{bmatrix}
  2 & 2 \\
  2 & 4 \end{bmatrix}.
\]
The first two lattices have truant $5$, and the third has truant $7$. The escalation of these lattices result in $50$ three-dimensional
lattices. Of these, $11$ have no truant below $1000$. In fact, all of these are relatively universal.

\begin{thm}
\label{ternaries}
The following $11$ ternary quadratic forms represent all positive integers coprime to $3$.
\begin{align*}
  & x^{2} - xy + y^{2} + z^{2}\\
  & x^{2} + xy + 2y^{2} + yz + 4z^{2}\\
  & x^{2} + xy + 2y^{2} - xz - 2yz + 4z^{2}\\
  & x^{2} + xy + 2y^{2} - xz + 2z^{2}\\
  & x^{2} + xy + 2y^{2} - xz + 5z^{2}\\
  & x^{2} + xy + 2y^{2} + 3z^{2}\\
  & x^{2} + y^{2} + xz - yz + 2z^{2}\\
  & x^{2} + y^{2} + 3z^{2}\\
  & x^{2} + y^{2} - xz + 4z^{2}\\
  & x^{2} + y^{2} + 6z^{2}\\
  & x^{2} + y^{2} - xz + 7z^{2}.
\end{align*}
\end{thm}
\begin{proof}
  It is easy to verify that all $11$ of these quadratic form locally
  represent all integers $n$ coprime to $3$. Of these 11 forms, 8 are
  in a genus of size $1$, and this automatically implies that they are
  regular. The remaining three are also regular (since they occur on
  the list of forms proven regular in \cite{JKS}). Since a regular
  form represents everything that is represented locally, and each
  form locally represents everything coprime to $3$, each of these
  $11$ forms represents all positive integers coprime to $3$.
\end{proof}

Escalating the remaining $39$ ternary quadratic forms gives rise to
$8894$ quadratic forms $Q$ in four variables. Of these $8884$ locally
represent everything coprime to $3$. (We refer to these as the
\emph{basic} quaternary lattices.)  The remaining $10$ fail to locally
represent all integers coprime to $3$.  These ten lattices all have
either $14$, $22$ or $35$ as their truant.  We take a ternary
sublattice of each, and escalate by the truant of the quaternary
lattice - this gives rise to $727$ \emph{extra} quaternary lattices
(all of which locally represent everything coprime to
$3$). Understanding the squarefree integers that are not represented
by these $8894+727-10 = 9611$ basic and extra quadratic forms suffices
to prove Theorem~\ref{main}.  The next several sections outline how we
understand these $9611$ quadratic forms.

\section{Regular ternaries}
\label{regtern}

Given a quaternary lattice $L$, the first thing we check is if $L$ has
a ternary sublattice $L'$ whose quadratic form is one of the $11$ listed
in Theorem~\ref{ternaries}. If so, then for every number $t$ coprime
to $3$, $L$ has a vector with norm $t$ and this means that the quadratic
form corresponding to $L$ represents $t$. Therefore it is relatively
universal.

\begin{ex}
Form 8703 is $Q(x,y,z,w) = x^{2} - xz + y^{2} - yz + 2z^{2} + 7w^{2}$.
We have that $Q(-y-z,x,-y+z,0) = x^{2} + xy + 2y^{2} - xz - 2yz + 4z^{2}$.
This is the third universal ternary listed in Theorem~\ref{ternaries}. Since
this ternary form represents everything coprime to $3$, so does $Q$.
\end{ex}

This first method applies to $644$ of the $8884$ basic quaternary lattices,
and $2$ of the $727$ extra lattices.

The second method takes a quaternary lattice $L$ and searches for a
ternary sublattice $K$ so that the quadratic form corresponding to $K$
is regular, and the quadratic form corresponding to
$K \oplus K^{\perp}$ locally represents everything coprime to $3$. We
may write the quadratic form corresponding to $K \oplus K^{\perp}$ as
$Q(x,y,z,w) = T(x,y,z) + dw^{2}$.  Since $T$ is regular, whether an integer $m$ is
represented by $T$ depends only on congruence conditions.

The paper \cite{JKS} proves that there are at most $913$ ternary quadratic
forms, and $891$ of those are definitely regular. In \cite{BOh}, another
$8$ ternary quadratic forms are proven regular (and in \cite{RLO}, the remaining $14$ are proven regular assuming the Generalized Riemann Hypothesis). We use
the list of $899$ provably regular ternary quadratic forms.

We let $M$ be a positive integer divisible by all primes dividing the determinant of $T$ so that for every $a$, either $T$ represents every squarefree integer
$\equiv a \pmod{M}$ or no integer $\equiv a \pmod{M}$.

We create a queue of residue classes to check, initially including
all residue classes modulo $M$ that contain integers coprime to $3$
that are not represented by $T$. If $a \pmod{M}$ is such a residue class,
there is some integer $n = T(x,y,z) + dw^{2} \equiv a \pmod{M}$ 
that is represented by $Q(x,y,z,w)$. If $T(x,y,z) = b \ne 0$,
there is some arithmetic progression $b \pmod{M'}$ of positive integers 
represented by $T(x,y,z)$ and hence every integer $\equiv a \pmod{M'}$ and
greater than or equal to $n$ is represented by $Q$.

If $M' = M$, then the only positive integers $\equiv a \pmod{M}$ not
represented by $Q$ are those less than $n$. If $M' > M$, we add the
residue classes $a + kM \pmod{M'}$ to the queue (for $0 < k < M'/M$).
We proceed until the queue is empty.

\begin{ex}
Consider Form 238, $Q(x,y,z,w) = x^{2} + y^{2} + yz + 2z^{2} - xw + 23w^{2}$ and
let $L$ be the corresponding quaternary lattice. There is a lattice $L'$
of index two in $L$ of the form $K \oplus K'$. The quadratic form
corresponding to $L'$ is $x^{2} + y^{2} + yz + 2z^{2} + 91w^{2}$.
The form $T(x,y,z) = x^{2} + y^{2} + yz + 2z^{2}$ is regular,
and represents every positive integer except those of the form
$7^{r} m$ where $r$ is odd and $\legen{m}{7} = -1$. Thus,
the modulus of $M$ is $49$, and initially our queue is set to include
the classes $21 \pmod{49}$, $35 \pmod{49}$ and $42 \pmod{49}$.

When we test the residue class $21 \pmod{49}$, we find that neither
$21$ nor $70$ is represented by $T(x,y,z) + 91w^{2}$. However,
$119 = 91 + 28$ is represented and since $Q$ represents everything
$\equiv 28 \pmod{49}$, all $n > 70$ with $n \equiv 21 \pmod{49}$
are represented represented by $Q$.

When we test the residue class $35 \pmod{49}$, we find that the
smallest positive integer in this residue class represented is
$378 = 91 \cdot 2^{2} + 14$. Since $T$ represents
everything $\equiv 14 \pmod{49}$, $Q$ represents everything $\equiv 35 \pmod{49}$
and greater than $378$.

When we test the residue class $42 \pmod{49}$ we find that $42$ is not
represented, that $91$ is represented but only as $91 \cdot 1^{2} + 0$.
We have $140 = 91+49$. Now, $T$ does not represent all numbers $\equiv 0 \pmod{49}$, but it does represent all those $\equiv 49 \pmod{343}$. This proves
that $Q$ represents all numbers $\equiv 140 \pmod{343}$, but we add to the
queue the six classes $42 \pmod{343}$, $91 \pmod{343}$, $189 \pmod{343}$,
$238 \pmod{343}$, $287 \pmod{343}$, and $336 \pmod{343}$. These
are easily checked.

We find in the end that $T(x,y,z) + 91w^{2}$ represents all positive
integers except $21$, $35$, $42$, $70$, $84$, $133$, $182$, $231$, $280$ and 
$329$. Testing $Q$, we find that it represents all positive integers coprime
to $3$ except $70$. 
\end{ex}

This method applies to $3465$ of the basic quaternary lattices, and
$166$ of the extra lattices.

\section{Modular forms}
\label{modform}

If $Q = \frac{1}{2} \vec{x}^{T} A \vec{x}$ is a positive-definite
quaternary quadratic form, then we use the theta series
\[
  \sum_{n=0}^{\infty} r_{Q}(n) q^{n} = \sum_{n=0}^{\infty} a_{E}(n) q^{n}
  + \sum_{n=1}^{\infty} a_{C}(n) q^{n}
\]
and enumerate all squarefree $n$ so that $a_{E}(n) \leq |a_{C}(n)|$.

The Eisenstein coefficient $a_{E}(n) = \prod_{p \leq \infty} \beta_{p}(n)$.
Formulas for local densities are known (see \cite{Hanke} and \cite{Yang})
and imply that if $n$ is squarefree, then
\[
  \beta_{p}(n) =
  \begin{cases}
    \frac{\pi^{2} n}{\sqrt{\det(A)}} & \text{ if } p = \infty\\
    1 - \frac{\chi_{D}(p)}{p^{2}} & \text{ if } p \nmid N \text{ and } p \nmid n\\ 
    \frac{(p-1)(p^{2} + (1+\chi_{D}(p))p + 1)}{p^{3}} & \text{ if }
  p \nmid N \text{ and } p | n.
\end{cases}
\]
Thus,
\begin{align*}
  a_{E}(n) &= \frac{\pi^{2} n}{\sqrt{\det(A)}} \left(\prod_{p | N} \beta_{p}(n)\right) 
\prod_{\substack{p \nmid N \\ p | n}} \frac{(p-1)(p^{2} + (1+\chi_{D}(p)) p + 1)}{p^{3}} \prod_{p \nmid Nn} \left(1 - \chi_{D}(p)/p^{2}\right)\\
  &= \frac{\pi^{2} n}{\sqrt{\det(A)}}
  \left(\prod_{p | N} \frac{\beta_{p}(n)}{1 - \chi_{D}(p)/p^{2}}\right)
  \left(\prod_{\substack{p \nmid N \\ p | n}}
  \frac{(p-1)(p^{2} + (1 + \chi_{D}(p))p + 1)}{p^{3} - \chi_{D}(p) p}\right)
  \prod_{p} \left(1 - \chi_{D}(p)/p^{2}\right)\\
  &\geq \frac{\pi^{2} n}{\sqrt{\det(A)} L(2,\chi_{D})}
  \left(\prod_{p | N} \frac{\beta_{p}(n)}{1 - \chi_{D}(p)/p^{2}}\right)
  \prod_{\substack{p \nmid N \\ p | n, \chi_{D}(p) = -1}} \left(\frac{p-1}{p+1}\right).
\end{align*}
We compute $\beta_{p}(n)$ for all primes $p | 2N$ and for all of the
different $\mathbb{Z}_{p}$ square classes containing squarefree
integers. Unlike the past work of Rouse \cite{451paper}, we do so by using the 
non-recursive formulas given in \cite{Yang}. These are more efficient
than the procedure given in \cite{Hanke}. In this
way, we compute a constant $C_{E}$ so that 
\[
  a_{E}(n) \geq C_{E} \prod_{\substack{p | n \\ p \nmid N, \chi_{D}(p) = -1}}
  \frac{p-1}{p+1}
\]
for all squarefree positive integers $n$.

Any cusp form $C(z)$ can be decomposed as
\[
  C(z) = \sum_{d | N} \sum_{i=1}^{s} \sum_{e | \frac{d}{{\rm cond}~\chi}} g_{i}(ez)
\]
where $g_{i}(z)$ is a normalized Hecke eigenform living in the new subspace
of $S_{2}(\Gamma_{0}(d), \chi)$. The $n$th Fourier coefficient of $g_{i}(z)$
has size at most $d(n) \sqrt{n}$. Thus, if we set
\[
  C_{Q} = \sum_{d | N} \sum_{i=1}^{s} \sum_{e | \frac{d}{{\rm cond}~\chi}} \frac{|c_{d,i,e}|}{\sqrt{e}},
\]
we have $|a_{C}(n)| \leq C_{Q} d(n) \sqrt{n}$. This implies that
there is a constant $F$ so that if
\[
  F_{4}(n) := \frac{\sqrt{n}}{d(n)} \prod_{\substack{p \nmid N, p | n \\ \chi_{D}(p) = -1}} \frac{p-1}{p+1} > F,
\]
then $n$ is represented by $Q$. 

In the third method, we will explicitly compute the newforms $g_{i}(z)$
(using the modular symbols algorithm provided in Magma \cite{Magma}) and
compute the constants $C_{Q}$ and $F$. This procedure is somewhat time consuming. Once we have computed $F$, we will enumerate all squarefree integers
with $F_{4}(n) \leq F$, and see which of these are represented by $Q$, and
which are not.

\begin{ex}
  Consider Form 8819,
  $Q(x,y,z,w) = x^{2} + y^{2} + 7z^{2} - xw - yw + 7zw + 12w^{2}$,
  with corresponding lattice $L$. This form has level $546$. The
  dimension of $S_{2}(\Gamma_{0}(546), \chi_{273})$ is $104$. We
  compute that $C_{E} = 12/37$, and that $C_{Q} \approx 23.925$. This
  yields that $F = 74.507$. Any squarefree $n$ with $F(n) \leq F$ has
  at most $8$ prime factors, all of which are $\leq 79939$. There are
  a total of $395007$ squarefree $n$ coprime to $3$ for which
  $F(n) \leq F$. The form
  $Q'(x,y,z,w) = w^{2} + 2x^{2} + 14y^{2} + 78z^{2}$ is the quadratic
  form corresponding to a sublattice of $L$.  We make an array of the
  values represented by $T(x,y,z) = 2x^{2} + 14y^{2} + 78z^{2}$ with
  $0 \leq x, y \leq 800$ and $z \leq 278$. For each of the $395007$
  squarefree integers $n$, we check to see if there is some integer
  $w$ so that $n-w^{2}$ is represented by $T(x,y,z)$. Of the $395007$
  squarefree $n$ coprime to $3$ with $F(n) \leq F$, at most $9$ are
  not represented by $Q'$, and $Q$ represents all positive integers
  coprime to $3$ except $19$, $22$, $31$, $35$ and $133$.
\end{ex}

A discriminant $D$ is an integer $\equiv 0 \text{ or } 1 \pmod{4}$.  A
\emph{fundamental discriminant} is a discriminant $D$ that is not a
square multiple of another discriminant. The method above is used for
all remaining quadratic forms $Q$ for which $\det(A)$ is not a fundamental
discriminant. There are $1906$ such basic quaternary lattices, and $361$
extra lattices. A few of these cases are quite time consuming. For
example, Form 3391 is handled with this method, and requires about
$22$ hours of computation time to compute the $g_{i}$ and the
constants $c_{d,i,e}$. This motivates an additional method for
computing an upper bound on $C_{Q}$, but not exactly computing it.

\section{Petersson inner products}
\label{petnorm}

For the remaining cases, we use the method introduced in \cite{451paper}.
If $f, g \in S_{2}(\Gamma_{0}(N), \chi)$, the Petersson inner product of
$f$ and $g$ is defined by
\[
  \langle f, g \rangle = \frac{3}{\pi [\SL_{2}(\Z) : \Gamma_{0}(N)]}
  \iint_{\H / \Gamma_{0}(N)} f(x+iy) \overline{g(x+iy)} \, dx \, dy.
\]
We will assume throughout this section that $\chi$ is a primitive Dirichlet
character modulo $N$. This means that the new subspace of $S_{2}(\Gamma_{0}(N),\chi)$ is the entire space, and that any cusp form $C(z)$ in this space 
has a decomposition
\[
  C(z) = \sum_{i=1}^{s} c_{i} g_{i}(z), \quad s = \dim S_{2}(\Gamma_{0}(N), \chi)
\]
where the $g_{i}$ are newforms. It is known that distinct newforms
are orthogonal with respect to the Petersson inner product, and hence
$\langle C, C \rangle = \sum_{i=1}^{s} |c_{i}|^{2} \langle g_{i}, g_{i} \rangle$.
Therefore, if $Q$ is a quadratic form and $\theta_{Q}(z) = E(z) + C(z)$,
we may bound $C_{Q}$ by finding positive $A$ and $B$ so that $\langle C, C \rangle \leq A$ and $\langle g_{i}, g_{i} \rangle \geq B$ for all $i$.
Then, the Cauchy-Schwarz inequality gives
\[
  C_{Q} = \sum_{i=1}^{s} |c_{i}| \leq \sqrt{s} \sqrt{\sum_{i=1}^{s} |c_{i}|^{2}}
  \leq \sqrt{\frac{As}{B}}.
\]

A newform $g_{i}$ is said to have \emph{complex multiplication} or CM
if $g_{i}$ arises from a Hecke Gr\"ossencharacter, or equivalently
if there is some negative integer $-D$ so that $\legen{-D}{p} = -1$
implies that the $p$th coefficient of $g_{i}$ is equal to zero.
If $g_{i}$ does not have complex multiplication, then Proposition~11
of \cite{451paper} proves that
\[
  \langle g_{i}, g_{i} \rangle \geq \frac{3}{208 \pi^{4} \prod_{p | N} (1+1/p) \log(N)}.
\]
For a given $Q$, we can also explicitly enumerate all CM forms in
$S_{2}(\Gamma_{0}(N), \chi)$ and verify that the same bound holds for those.

To compute an upper bound on $\langle C, C \rangle$, we use the following
method. If $\epsilon \in \{ \pm 1 \}$, define
\[
  S_{2}^{\epsilon}(\Gamma_{0}(N), \chi) = \{ f \in S_{2}(\Gamma_{0}(N),\chi) :
  \text{ if } f = \sum a(n) q^{n}, \text{ then } a(n) = 0 \text{ if }
  \chi(n) = -\epsilon \}.
\]
Instead of working directly with $Q = \frac{1}{2} \vec{x}^{T} A \vec{x}$, 
we work with $Q^{*} = \frac{1}{2} \vec{x}^{T} NA^{-1} \vec{x}$. Let
$\theta_{Q^{*}} = E^{*}(z) + C^{*}(z)$ be the decomposition into Eisenstein
series and cusp forms. Proposition 15
of \cite{451paper} shows that $\langle C, C \rangle = N \langle C^{*}, C^{*} \rangle$ and also that $C^{*} \in S_{2}^{-}(\Gamma_{0}(N), \chi)$. Finally,
if $f(z) = \sum a(n) q^{n} \in S_{2}^{-}(\Gamma_{0}(N), \chi)$,
Proposition 14 of \cite{451paper} gives a formula for $\langle f, f \rangle$.
Let $\psi(x) = -\frac{6}{\pi} K_{1}(4 \pi x) + 24 x^{2} K_{0}(4 \pi x)$,
where $K_{0}$ and $K_{1}$ are the usual $K$-Bessel functions. Then
\[
  \langle f, f \rangle
  = \frac{1}{[\SL_{2}(\Z) : \Gamma_{0}(N)]}
  \sum_{n=1}^{\infty} \frac{2^{\omega(\gcd(n,N))} a(n)^{2}}{n} \sum_{d=1}^{\infty} \psi\left(d \sqrt{\frac{n}{N}}\right).
\] 
We compute the first $15N$ coefficients of $C^{*}$ and let
$C_{1}$ be the result of adding the terms above for $1 \leq n \leq 15N$.
A bound on $\langle C^{*}, C^{*} \rangle$ translates into a bound on $a(n)$.
Using this, we obtain the inequality
\[
  \langle C^{*}, C^{*} \rangle
  \leq C_{1} + 1.71 \cdot 10^{-18} \frac{N^{7/4} s}{B} \left(1 + \frac{1}{15N}\right)^{3/2} \langle C^{*}, C^{*} \rangle.
\]
If the coefficient of $\langle C^{*}, C^{*} \rangle$ on the right hand
side is less than $1$, we can solve the inequality and obtain an upper
bound on $\langle C^{*}, C^{*} \rangle$ and hence on
$\langle C, C \rangle$. (The smallest $N$ for which the coefficient is
larger than $1$ is $44520$, and all of the forms we work with have
level $\leq 4292$.) Combining this with the lower bound on $\langle g_{i}, g_{i} \rangle$ for all $i$, this gives an upper bound on $C_{Q}$.

\begin{ex}
  Form 3995 is
  $Q(x,y,z,w) = x^{2} + 2y^{2} - xz + yz + 5z^{2} - yw + 29w^{2}$ and
  has $\det(A) = N = 4273$, a fundamental discriminant. We have that
  $s = \dim S_{2}(\Gamma_{0}(4273), \chi_{4273}) = 354$. There are no
  newforms with CM in this space, and $B = 2.66538 \cdot 10^{-5}$ is a
  lower bound for $\langle g_{i}, g_{i} \rangle$ for all $i$. Using
  the method described above, we find that
  $0.01297 \leq \langle C, C \rangle \leq 0.013026$ and this gives
  $C_{Q} \leq 415.9506$. It follows that if $n$ is squarefree and
  coprime to $3$ and $F_{4}(n) > 996.385$, the $n$ is represented by
  $Q$. This calculation requires $77$ seconds.

There are precisely $193766918$ squarefree integers coprime to $3$ for
which $F_{4}(n) \leq 996.385$ and the only one of these that is not
represented by $Q$ is $37$. This calculation takes $177$
seconds. Therefore, $Q$ represents every positive integer coprime to
$3$ except $37$.
\end{ex}

We use this method on all of the remaining forms. This accounts for $2869$ basic
quaternary forms, and $198$ extra quaternary forms.

\section{Proofs}
\label{pfs}

\begin{proof}[Proof of the CCXC Theorem]
If $Q$ is a positive-definite quadratic form that represents the
numbers in the statement of Theorem~\ref{main}, then the corresponding
lattice $L$ contains either (i) a relatively universal ternary lattice,
(ii) one of the $8884$ basic quaternary lattices, or (iii) one
of the $10$ quaternary lattices that fails to locally represent everything
coprime to 3.

In the first case, $Q$ is relatively universal. 

In the second case, examining the squarefree positive integers not
represented by the basic quaternary forms shows that the only integers
coprime to $3$ that might (i) not be represented by $Q$, and (ii) not
be in the statement of Theorem CCXC are $133$ and $187$. This is because 
form $6841$ has truant $17$ but fails to represent $187$, and
form
$8819$ has truant $19$ but fails to represent $133$. We compute all escalations
of form $6821$ (there are $22310$) and none fails to represent $187$.
We compute all escalations of form $8819$ (there are $20184$) and find that each of
them represents $133$. Thus, $Q$ is relatively universal.

In the third case, we must examine the extra lattices. Among the $727$
extra lattices, we find forms that fail to represent $86$, $91$, 
$133$, $139$, $142$, $154$, $166$, $182$, $214$, $238$, $266$,
$287$, $322$, $329$, $406$ and $434$. We escalate each of the $10$
forms that fail to locally represent all integers coprime to $3$
by their truant, and check each escalation to see if it fails to
represent any of these numbers. All escalations of these $10$ basic
lattices represent all integers coprime to $3$ except those in the statement
of the theorem, and this shows that if $L$ contains one of these $10$ lattices,
then $Q$ must be relatively universal. This completes the proof. The Magma
scripts used and log files are available at \url{http://users.wfu.edu/rouseja/CCXC/}.
\end{proof}

\begin{rem}
There are only two quadratic forms in our list that fail to represent $119$:
form $379$ and form $8891$ (one of those that fails to locally represent all
integers coprime to $3$). Form $379$ has truant $70$ and all $264341$ 
escalations of this form by $70$ represent $119$. Form $8891$ has a unique
escalation by its truant $35$ that fails to represent $119$,
namely $x^{2} + y^{2} + 7z^{2} + 7zw + 14w^{2} + 35v^{2}$.
\end{rem}

\begin{proof}[Proof of Corollary~\ref{necessary}]
Let $t$ be one of the positive integers in the statement of the CCXC Theorem.
Then there is a form $Q(\vec{x})$ that has truant $t$ (see Appendix~A). The form
\[
  Q(\vec{x}) + (t+1)(a^{2} + b^{2} + c^{2} + d^{2}) + \sum_{i=1}^{m-1}
  (t+1+i) x_{i}^{2}
\]
can be easily seen (by Lagrange's four square theorem) to represent
every positive integer larger than $t$, and every positive integer coprime to $3$ less than $t$. This proves the desired claim.
\end{proof}

\begin{proof}[Proof of Corollary~\ref{corollary}]
  We escalate the one-dimensional lattice with Gram matrix $[2]$
  repeatedly, taking only integer-matrix escalations. We use the CCXC
  Theorem to determine if our escalator lattices are relatively
  universal, and proceed until we have found relatively universal
  escalator lattices.  There are two binary escalator lattices
  corresponding to $x^{2} + y^{2}$ and $x^{2} + 2y^{2}$ with truants
  $7$ and $5$, respectively.  Escalatings these two gives $12$ ternary
  lattices of which two are relatively universal.

Escalating the other $10$ gives rise to $261$ quadratic lattices,
of which all but $15$ are relatively universal. All but two of the
escalations of these 15 quaternary lattices are universal. These
correspond to $x^{2} + y^{2} + 4z^{2} + 4zv + 7w^{2} + 13v^{2}$
and $x^{2} + y^{2} + 7z^{2} + 13w^{2} + 12wv + 13v^{2}$. Of the integers
relatively prime to $3$ and less than $290$, these fail to represent only
$35$. Hence, all escalations will represent all integers coprime to $3$
less than $290$ and so will be relatively universal.
\end{proof}

\vfill\eject
\appendix
\section{Table of quadratic forms with given truants}
\label{truant_table}

\small
\begin{center}
\begin{tabular}{l|c}
Form & Truant\\
\hline
$2x^{2}$ & $1$\\
$x^{2}$ & $2$\\
$x^{2} + 2y^{2}$ & $5$\\
$x^{2} + y^{2}$ & $7$\\
$x^{2} -xy + y^{2} + 2z^{2} + 2zw + 4w^{2}$ & $10$\\
$x^{2} + y^{2} + 5z^{2}$ & $11$\\
$x^{2} + 2y^{2} - xz + 2yz + 5z^{2} - xw - 2zw + 4w^{2}$ & $13$\\
$x^{2} + y^{2} + 2z^{2} - yw + 3w^{2}$ & $14$\\
$x^{2} + xy + 2y^{2} - xz - yz + 4z^{2} + zw + 5w^{2}$ & $17$\\
$x^{2} + xy + 2y^{2} - xz - yz + 3z^{2} - xw - 2zw + 6w^{2}$ & $19$\\
$x^{2} + y^{2} - xz + yz + 7z^{2} - 2zw + 10w^{2}$ & $22$\\
$x^{2} + 2y^{2} + yz + 5z^{2} - xw + 7w^{2}$ & $23$\\
$x^{2} + 2y^{2} + 3z^{2} - 2yw + 3zw + 9w^{2}$ & $26$\\
$x^{2} + 2y^{2} - xz - yz + 4z^{2}$ & $29$\\
$x^{2} + xy + 2y^{2} - xz + yz + 5z^{2} + 2zw + 10w^{2}$ & $31$\\
$x^{2} + xy + 2y^{2} - xz + yz + 3z^{2} + 17w^{2}$ & $34$\\
$x^{2} + 2y^{2} + 5z^{2} - xw - 2yw + zw + 10w^{2}$ & $35$\\
$x^{2} + 2y^{2} - xz + yz + 5z^{2} - xw - 2zw + 12w^{2}$ & $37$\\
$x^{2} + y^{2} - xz + 5z^{2} - xw - yw - 3zw + 11w^{2}$ & $38$\\
$x^{2} + y^{2} - xz + yz + 5z^{2} - xw + 5zw + 11w^{2}$ & $46$\\
$x^{2} + y^{2} + yz + 6z^{2} + 9w^{2}$ & $47$\\
$x^{2} + y^{2} + 7z^{2} - xw + 7zw + 8w^{2}$ & $55$\\
$x^{2} + 2y^{2} + 3z^{2} - xw - yw + 3zw + 8w^{2}$ & $58$\\
$x^{2} + xy + 2y^{2} + 5z^{2} - wz + 5zw + 6w^{2}$ & $62$\\
$x^{2} + y^{2} + yz + 2z^{2} - xw + 23w^{2}$ & $70$\\
$x^{2} + y^{2} - xz + yz + 5z^{2} - xw - yw + 12w^{2}$ & $94$\\
$x^{2} + y^{2} + yz + 3z^{2} + 22w^{2}$ & $110$\\
$x^{2} + y^{2} + 7z^{2} + 7zw + 14w^{2} + 35v^{2}$ & $119$\\
$x^{2} + 2y^{2} - xz - yz + 4z^{2} + 29w^{2}$ & $145$\\
$x^{2} + 2y^{2} - xz - yz + 4z^{2} + 29w^{2} + 29zv + 58wv + 145v^{2}$ & $203$\\
$x^{2} + 2y^{2} - xz - yz + 4z^{2} + 29w^{2} + 29zv + 87wv + 145v^{2}$ & $290$\\
\end{tabular}
\end{center}
\normalsize

\bibliographystyle{amsplain}
\bibliography{290refs}

\end{document}